\documentclass[oneside,english]{elsarticle}
\usepackage[T1]{fontenc}
\usepackage[latin9]{inputenc}
\usepackage{amsthm}
\usepackage{setspace}
\usepackage{amssymb}
\usepackage{graphicx}
\usepackage{babel}
\usepackage{verbatim}
\usepackage{amsmath,amsthm,amssymb,amsfonts,graphics, epsfig}
\theoremstyle{plain}
\newtheorem{thm}{Theorem}
\newtheorem{lem}{Lemma}
\newtheorem{clry}{Corollary}

\theoremstyle{remark}

\theoremstyle{definition}
\newtheorem{defn}{Definition}
\numberwithin{equation}{section}
\numberwithin{figure}{section}

\begin{document}
\begin{frontmatter}
\title{Irreversible $k$-threshold and majority conversion processes \\
on complete multipartite graphs and graph products}

\author[label1]{Sarah Spence Adams}
\author[label1]{Zachary Brass}
\author[label1]{Connor Stokes}
\author[label2]{Denise Sakai Troxell}
\address[label1]{Franklin W. Olin College of Engineering, Needham, MA 02492}
\address[label2]{Corresponding Author, Babson College, Babson Park, MA 02457, troxell@babson.edu}

\begin{abstract}
In graph theoretical models of the spread of disease through populations, the spread of opinion through social networks, and the spread of faults through distributed computer networks, vertices are in two states, either black or white, and these states  are dynamically updated at discrete time steps according to the rules of the particular conversion process used in the model. This paper considers the irreversible $k$-threshold and majority conversion processes. In an irreversible $k$-threshold (resp., majority) conversion process, a vertex is permanently colored black in a certain time period if at least $k$ (resp., at least half) of its neighbors were black in the previous time period. A $k$-conversion set (resp., dynamic monopoly) is a set of vertices which, if initially colored black, will result in all vertices eventually being colored black under a $k$-threshold (resp., majority) conversion process.  We answer several open problems by presenting bounds and some exact values of the minimum number of vertices in $k$-conversion sets and dynamic monopolies of complete multipartite graphs, as well as of Cartesian and tensor products of two graphs. 
\end{abstract}


\end{frontmatter}

\section{Introduction}
Graph theoretical models have been  used in the past decade to study the spreads of (a) disease through a population, (b) opinion through a social network, and (c) faults in a distributed computer network \cite{products:triangle, products:brief, products:dreyerroberts, products:tori, products:competition, products:np, products:consensus, products:majority}. In these models, each vertex represents a person/object that is either colored black (\emph{e.g.}, infected, holding a certain opinion, corrupted) or white (\emph{e.g.}, uninfected, not holding a certain opinion, not corrupted),  and each edge represents a relevant interaction between two people/objects.  At discrete time steps, the colors of the vertices are updated under some conversion process.  Many natural questions then arise.  For example, what is the minimum number of initially black vertices in a graph that will eventually cause all vertices to turn black? Or, given an initially black subset of vertices in a graph, how many discrete time steps will it take to turn all of its vertices black? These questions may have relevance in applications to the design of immunization and containment strategies to prevent the spread of a disease, opinion, or fault to all of the vertices in the graph.  While many other related questions can be asked, this paper focuses on the first question under two distinct conversion processes.

In an \emph{irreversible $k$-threshold conversion process}, if a white vertex has at least $k$ black neighbors at time $t-1$, it will be permanently colored black at time $t$. Irreversible $k$-threshold conversion processes have been used to model the spread of disease and opinion through social networks \cite{ products:brief, products:dreyerroberts, products:np}.  A \emph{$k$-conversion set} is a set of vertices that, when colored black at time $t=0$, will eventually cause the graph to turn completely black under an irreversible $k$-threshold conversion process.  In this paper, we consider the fundamental problem of determining exact values or upper bounds for the minimum number of vertices in $k$-conversion sets of a graph $G$. This number is denoted by $min_k(G)$. A $k$-conversion set with exactly $min_k(G)$ vertices will be called \emph{minimum $k$-conversion set}. Previously, exact values of $min_k(G)$ for all possible $k$ have been found when $G$ is a path, cycle, complete multipartite graph, or tree \cite{products:dreyerroberts}. In the same article, some exact values and upper bounds of $min_k(G)$ for certain values of $k$ have also been determined when $G$ is a rectangular, cylindrical, or toroidal square grid.

In an \emph{irreversible majority conversion process}, if at least half of the neighbors of a white vertex are black at time $t-1$, the vertex will be permanently colored black at time $t$. Although irreversible majority conversion processes may be used to model the spread of opinion through a social network, these processes are more often used to model the propagation of computer faults \cite{products:triangle, products:constant, products:monotone, products:tori, products:consensus, products:majority, products:voting, products:bounds}.  For example, in a distributed computer network, nodes communicate with each other to maintain uncorrupted copies of the same variables. These nodes may determine the state of a variable based on a majority vote of their neighbors' states of the same variable. A set of vertices that, when colored black at time $t=0$, will eventually cause the entire graph to be colored black under an irreversible majority conversion process is called a \emph{dynamic monopoly}, or \emph{dynamo} \cite{products:triangle, products:constant, products:lyuu, products:monotone, products:tori, products:bounds}. We are concerned with such sets of minimum number of vertices, called \emph{minimum dynamos}, and we let $min_D(G)$ denote the number of vertices in a minimum dynamo of a graph $G$.  Exact values of $min_D(G)$ have been found when $G$ presents certain toroidal square grid structures \cite{products:tori}, while upper bounds of $min_D(G)$ have been found when $G$ is a planar, cylindrical, or toroidal triangular grid \cite{products:triangle}.

As both of the conversion processes considered in this paper are irreversible, that is, once a white vertex is colored black it will remain permanently black, we will omit the word \emph{irreversible} and we will often use \emph{color(ed)} (resp., \emph{uncolor(ed)}) to mean \emph{color(ed) black} (resp., \emph{color(ed) white}) for brevity.  Let $S$ (resp., $D$) be a $k$-conversion set (resp., dynamo) of a graph $G$ with set of vertices $V(G)$. We will sometimes simply say \emph{initially color} $S$ (resp., $D$) to mean that all the vertices in $S$ (resp., $D$) are initially colored black at time $t=0$ and all the vertices in $V(G)-S$ (resp., $V(G)-D$) are white.  We will also say that $S$ (resp., $D$) \emph{colors} $G$ or, equivalently, that $G$ \emph{is colored by} $S$ (resp., $D$) in $T$ time steps, if when $S$ (resp., $D$) is initially colored, all the vertices in $V(G)$ become colored by time step $T$ under a $k$-threshold (resp., majority) conversion process.%, and $T$ is minimal with this property.

In Section~\ref{m}, we determine exact values of $min_k(G)$ and $min_D(G)$ when $G$ is a complete multipartite graph, building on prior work by Dreyer and Roberts \cite{products:dreyerroberts}.  In Section~\ref{cart}, we provide upper bounds of the same graph invariants when $G$ is a Cartesian products of two graphs.  Section \ref{tensor} contains analogous results for the tensor products of two graphs, and in this case, the upper bounds are shown to be tight.  Conclusions and a discussion of future work are included in Section \ref{conclude}.

\section{Exact values of $min_k(G)$ and $min_D(G)$ of complete multipartite graphs $G$}\label{m}
For $m\geq 2$, let us denote by $K_{p_1,p_2, \ldots, p_m}$ the \emph{complete multipartite graph} with $m$ partite sets with $p_1, p_2, \ldots, p_m$ vertices, respectively.  Such graphs may be useful in modeling computer networks wherein most pairs of computers are connected, however certain subsets of computers share no connections. Dreyer and Roberts~\cite{products:dreyerroberts} provided exact values of $min_k(K_{p_1,p_2, \ldots, p_m})$ through a series of related results.   Here, we provide an equivalent unified result with a straight-forward statement and a simpler proof. For the remainder of this paper, we will use the standard notation $|X|$ for the cardinality of a set $X$.

\begin{thm}
Let $G=K_{p_1,p_2,...p_m}$ be a complete multipartite graph  with $p_1+p_2+\ldots+p_m=n $ and $p_1\geq p_2 \geq \ldots \geq p_m$. Let $X$ be the set of all vertices with degree less than $k$.
Then
\[ min_k(G)=\left\{\begin{array}{ll}
max\{|X|, k\} &\textrm{if } n> k\\
n &\textrm{if }n\leq k.\\
\end{array}
\right.\]
\label{thm:kmpartite}
\end{thm}
\begin{proof}
For each $i=1,2,...,m$, let $V_i$ be the partite set of $G$ with $p_i$ vertices. Note that a vertex belongs to $X$ if and only if its entire partite set belongs to $X$. Therefore $X$ is the union of some partite sets in $G$. Since all vertices in $X$ have degree less than $k$, $X$ must be contained in any $k$-conversion set, and therefore $|X| \leq min_k(G)$. If $n\leq k$, then $X=V(G)$ and therefore $min_k(G)=n$. If $k\leq |X|$, then $X$ is a minimum $k$-conversion set since $X$ is a union of whole partite sets and therefore all the vertices in $V(G) - X$ are adjacent to the vertices in $X$ (of which there are $\geq k$), so they would become colored in time step $t=1$ if $X$ is initially colored; hence $min_k(G)=|X|$.

Finally, we examine the case where $n>k>|X|$. To show that $min_k(G)= k$ under these conditions, it suffices to exhibit a $k$-conversion set $S$ with $k$ vertices since the inequality $min_k(G)\geq k$ is obvious. Color all vertices in $X$, and in addition, color a subset of $V(G)-X$ with $k-|X|$ vertices, where this subset is chosen to contain a union of entire partite sets, along with a possibly empty proper subset of some partite set $V_j$. This set $S$ of $k$ initially black vertices is non-empty because $n-|X|>k-|X|>0$. Moreover, $S$ is a union of partite sets, except possibly for the proper subset of $V_j$ because $X$ is the union of partite sets. Every vertex in $V(G)$, except possibly for the vertices in $V_j$, will be colored by time $t=1$ since every white vertex not in $V_j$ is connected to the $|X|$ vertices in $X$ and the $k-|X|$ vertices in $S-X$, giving a total of $k$ vertices. The only possible vertices that could still be white at time $t=1$ would be in $V_j$, but these would be colored at time $t=2$. This is because the vertices in $V_j$ are connected to every vertex outside $V_j$, which are all black by time $t=1$, and since $V_j$ is not part of $X$, its vertices have degree $\geq k$.  Thus, $S$ is a $k$-conversion set of size $k$ as desired.
\end{proof}

We now extend this previous result to a majority conversion process and find $min_D(K_{p_1,p_2, \ldots, p_m})$.

\begin{thm}
 Let $G=K_{p_1,p_2, \ldots, p_m}$ be a complete multipartite graph with $p_1+p_2+\cdots+p_m=n$ and $p_1\geq p_2\geq  \ldots \geq p_m$.  Then,  $min_D(G)=\left\lceil\dfrac{n-p_1}{2}\right\rceil$.
\label{thm:mpartite}
\end{thm}
\begin{proof}
%For each $i=1, 2, \ldots, m$, let $V_i$ be the partite set of $G$ with $p_i$ vertices; 
Note that the vertices in the partite set $V_i$ with $p_i$ vertices have degree $n-p_i \geq n-p_1$. So any dynamo of $G$ must have at least $\left\lceil\dfrac{n-p_1}{2}\right\rceil$ vertices, and therefore $min_D(G)\geq\left\lceil\dfrac{n-p_1}{2}\right\rceil$. To show that $min_D(G) = \left\lceil\dfrac{n-p_1}{2}\right\rceil$, it suffices to exhibit a dynamo with $\left\lceil\dfrac{n-p_1}{2}\right\rceil$ vertices. To do this, initially color a subset $D$ of $V(G)-V_1$ with $\left\lceil\dfrac{n-p_1}{2}\right\rceil$ vertices. Since the vertices in $V_1$ have $n-p_1$ neighbors, all vertices in $V_1$ will become colored at time $t=1$. Therefore, at least $p_1+\left\lceil\dfrac{n-p_1}{2}\right\rceil=\left\lceil\dfrac{n+p_1}{2}\right\rceil$ vertices will be colored by time $t=1$. At this point, since any partite set $V_i$ for $i=2,3,\ldots ,m$ has at most $p_i$ black vertices, any vertex in any $V_i$ is adjacent to at least $\left\lceil\dfrac{n+p_1}{2}\right\rceil - p_i$ black vertices; since $\left\lceil\dfrac{n-p_i}{2}\right\rceil$ vertices are required to color any vertex in $V_i$ (because a vertex in $V_i$ has $n-p_i$ neighbors), all uncolored vertices will become colored at time $t=2$ as $\left\lceil\dfrac{n+p_1}{2}\right\rceil - p_i =\left\lceil\dfrac{n+p_1-2p_i}{2}\right\rceil\geq \left\lceil\dfrac{n-p_i}{2}\right\rceil$. Hence, $D$ is a dynamo with $\left\lceil\dfrac{n-p_1}{2}\right\rceil$ vertices as desired.
\end{proof}

\section{Upper bounds of $min_k(G\square H)$ and $min_D(G\square H)$  of the Cartesian product $G\square H$ }\label{cart}
In this section, we will provide general upper bounds of $min_k(G\square H)$ and $min_D(G\square H)$ where $G\square H$ denotes the Cartesian product of two graphs $G$ and $H$. Let us first recall the definition of $G\square H$ and some related concepts.
\begin{defn}
\label{defn:Cartesian}
The \emph{Cartesian product} of two disjoint graphs $G$ and $H$ is the graph $G\square H$ such that $V(G\square H)=V(G)\times V(H)$ and two vertices $(u,u')$ and $(v,v')$ are adjacent if and only if either $u=v$ and $u'$ is adjacent to $v'$ in $H$, or $u'=v'$ and $u$ is adjacent to $v$ in $G$. For a fixed vertex $v$ in $H$, let $G_{v}$ be the subgraph of $G\square H$ induced by the vertices $(u,v)$ for every vertex $u$ in $G$; note that $G_{v}$ is isomorphic to $G$. Similarly, for a fixed vertex $u$ in $G$, let $H_{u}$ be the subgraph of $G\square H$ induced by the vertices $(u,v)$ for every vertex $v$ in $H$; note that $H_{u}$ is isomorphic to $H$.
\end{defn}

We first provide an upper bound of $min_k(G\square H)$.

\begin{thm}
Let $G$ and $H$ be two graphs. Then, $min_k(G \square H)\leq min_k(G)min_k(H)$.
\label{thm:cart_k-conv}
\end{thm}

\begin{proof}
Let $S_G$ and $S_H$ be minimum $k$-conversion sets of $G$ and $H$, respectively. Let $S$ be the set of $min_k(G)min_k(H)$ vertices $(u,v)$ in $G \square H$ where  $u\in{S_G}$ and $v\in{S_H}$. To verify the proposed upper bound, it suffices to show that $S$ is a $k$-conversion set of $G\square H$.

Initially color $S$. For a fixed vertex $u$ in $S_G$, $H_u$ contains vertices $(u,v)$ for all $v$ in $S_H$. Since these vertices are initially colored and $S_H$ is a $k$-conversion set of $H$, then all vertices in $H_u$ will eventually be colored as $H_u$ is isomorphic to $H$ and the coloration of $H$ by $S_H$ induces the coloration of $H_u$ by the set of vertices $(u,v)$ for $v$ in $S_H$.

To conclude the proof, it suffices to show that the union of all $H_u$ for $u$ in $S_G$ is also a $k$-conversion set of $G\square H$. Initially color the vertices in this union. For each $v$ in $H$, $G_v$ contains vertices $(u,v)$ for all $u$ in $S_G$. Since all the vertices in $H_u$ for $u$ in $S_G$ are colored and since $S_G$ is a $k$-conversion set of $G$, then all the vertices in $G_v$ will eventually get colored because $G_v$ is isomorphic to $G$ and the coloration of $G$ by $S_G$ induces the coloration of $G_v$ by the set of vertices $(u,v)$ for $u$ in $S_G$. Therefore, we showed that all vertices in $G\square H$ will get colored and consequently $S$ is a $k$-conversion set of $G\square H$ as desired.
\end{proof}

We will now provide the result analogous to Theorem~\ref{thm:cart_k-conv} under a majority conversion process. In what follows, given a graph $G$ and a non-empty subset $X$ of $V(G)$, we denote by $deg_G(u)$ the degree of a vertex $u$ in $G$, and define $deg_X(u)$ as the number of  neighbors of $u$ in $X$.

\begin{thm}
Let $G$ and $H$ be two graphs. Then,
\[ min_{D}(G \square H) \leq min_{D}(G)|V(H)| + min_{D}(H)|V(G)| - min_{D}(G)min_{D}(H). \]
\label{thm:Cartesian}
\end{thm}

\begin{proof}
Let $D_{G}$ and $D_{H}$ be two minimum dynamos of $G$ and $H$, respectively. Let $D$ be the set of vertices $(u,v)$ in $G\square H$ where either $u$ is in $D_{G}$ or $v$ is in $D_H$, or both. (Refer to Fig. \ref{fig:Cartesianproduct} for a concrete example).  Note that $D$ has $min_{D}(G)\left|V(H)\right| + min_{D}(H)\left|V(G)\right| - min_{D}(G)min_{D}(H)$ vertices:  $\left|V(H)\right|$ vertices $(u,v)$ for each $u$  in $D_{G}$, $\left|V(G)\right|$ vertices $(u,v)$ for each $v$  in $D_{H}$, and we must subtract $min_{D}(G)min_{D}(H)$ to account for double-counting the vertices $(u,v)$ where $u$ and $v$ are in $D_{G}$ and $D_{H}$, respectively. Furthermore, $min_D(G) {\left|V(H)\right|}+min_D(H) {\left|V(G)\right|}-min_D(G)min_D(H)$ $\geq$ $min_D(G)min_D(H)+min_D(H)min_D(G)-min_D(G)min_D(H)=min_D(G)min_D(H)> 0$, and therefore, $D\neq\emptyset$. To verify the proposed bound, it suffices to show that $D$ is a dynamo of $G\square H$.

First, we need to partition the vertices of $G\square H$ based on the coloration of $G$ by $D_G$ as follows. Suppose it takes $T$ time steps to fully color $G$ by $D_G$ and for each $t=0, 1, \ldots, T$, let $H(t)$ be the subgraph of $G\square H$ induced by the vertices $(u,v)$ where $u$ is a vertex that becomes colored at time step $t$ in $G$. Note that each $H(t)$ is the disjoint union of subgraphs $H_{u}$ (recall Definition~\ref{defn:Cartesian}) for $u$ in $H(t)$, and $H(0), H(1), \ldots, H(T)$ partitions the vertices in $G\square H$. (Refer to Fig.~\ref{fig:Gsteps} for the coloration of the individual graphs $G$, $H$, and $G\square H$ in Fig.~\ref{fig:Cartesianproduct}, and to Fig.~\ref{fig:Cartesian3} for the corresponding subgraphs $H(t)$ for $t=0, 1, \ldots, T$.)

\begin{figure}
\begin{center}
\includegraphics[width=120mm]{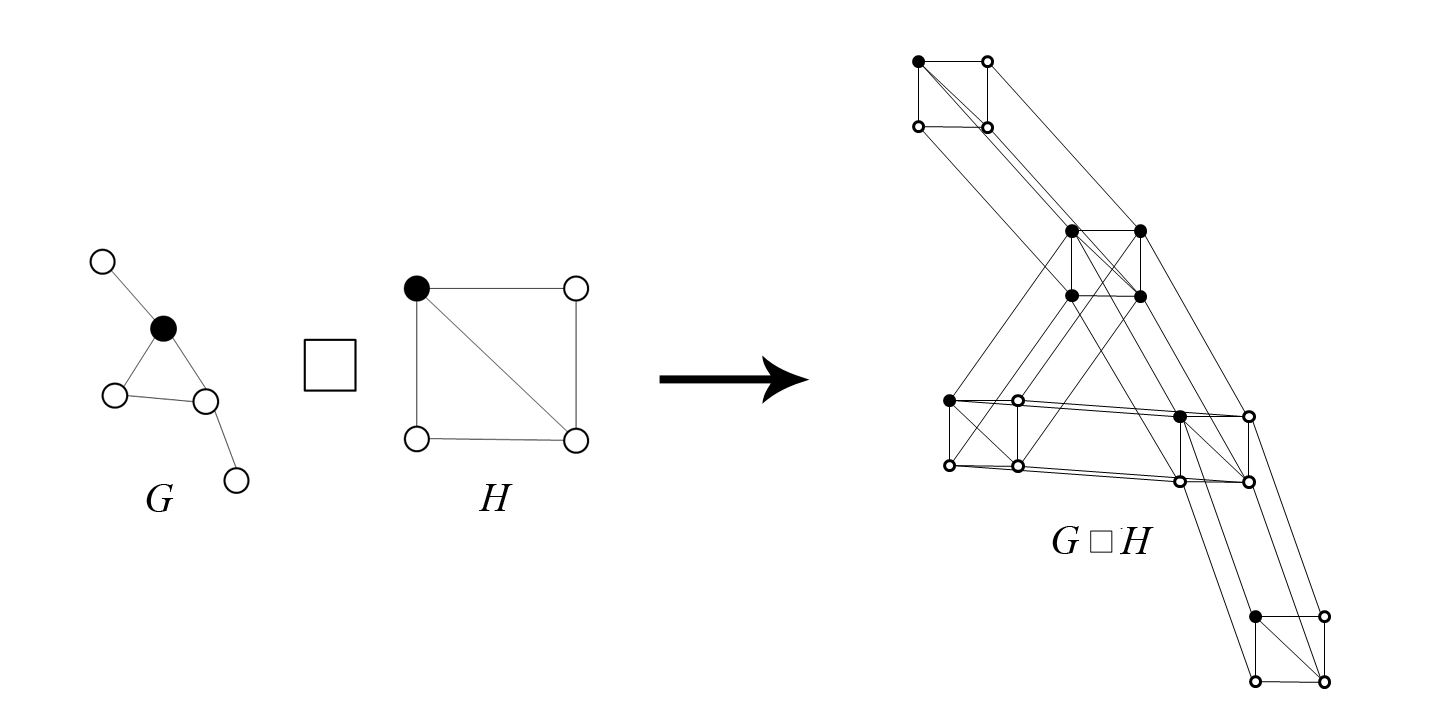}
\end{center}
\caption{Graphs $G$, $H$,  $G\square H$ with dynamos $D_G$, $D_H$, $D$, resp. (dynamos are colored black), where $\left| D\right| = min_D(G) {\left|V(H)\right|}+ min_D(H) {\left|V(G)\right|}- min_D(G)min_D(H)$ as constructed in the proof of Theorem~\ref{thm:Cartesian}.}
\label{fig:Cartesianproduct}
\end{figure}

Initially color $D$. We will show that all vertices in $G \square H$ will eventually be colored by proving that each $H(t)$ for $t=0, 1, \ldots, T$ will eventually become colored. We proceed by induction on $t$. As the base case, note that every vertex in $H(0)$ has its first coordinate in $D_{G}$ and consequently belongs to $D$ which is assumed to be initially colored. Now assume that all of the vertices in $H(0), H(1), \ldots, H(t-1)$ have been colored and let us show that the vertices in $H(t)$ will eventually become colored. Consider an arbitrary uncolored vertex $(u,v)$ in $H(t)$. Note that $(u,v)$ has $deg_{(G\square H)-H_u}(u,v)=deg_G(u)$ neighbors not in $H_u$ and $deg_{ H_u}(u,v)=deg_H(v)$ neighbors in $H_u$ and therefore $deg_{(G\square H)}(u,v)=deg_G(u)+deg_H(v)$. At least half of the neighbors of $(u,v)$ which are not in $H_{u}$ must be in the union of $H(i), i=0,1,\ldots , t-1$ which contains only colored vertices. Moreover, the set $(u,v')$ with $v'$ in $D_{H}$ is a dynamo of the subgraph $H_{u}$ so half of the neighbors of $(u,v)$ in this subgraph will eventually get colored similarly and at least as fast as $H$ is colored by $D_{H}$. Combining these colorations of the neighbors of $(u,v)$ in $H_{u}$ and out of $H_{u}$ we can conclude that half of all the neighbors of $(u,v)$ will eventually be colored, and consequently $(u,v)$ will become colored. Thus every vertex in $H(t)$ will be eventually colored. It now follows by induction that $D$ is a dynamo of $G \square H$  providing the desired upper bound of $min_D(G\square H)$.
\end{proof}

\begin{figure}[ht]
\centerline{\mbox{\includegraphics[width=69mm]{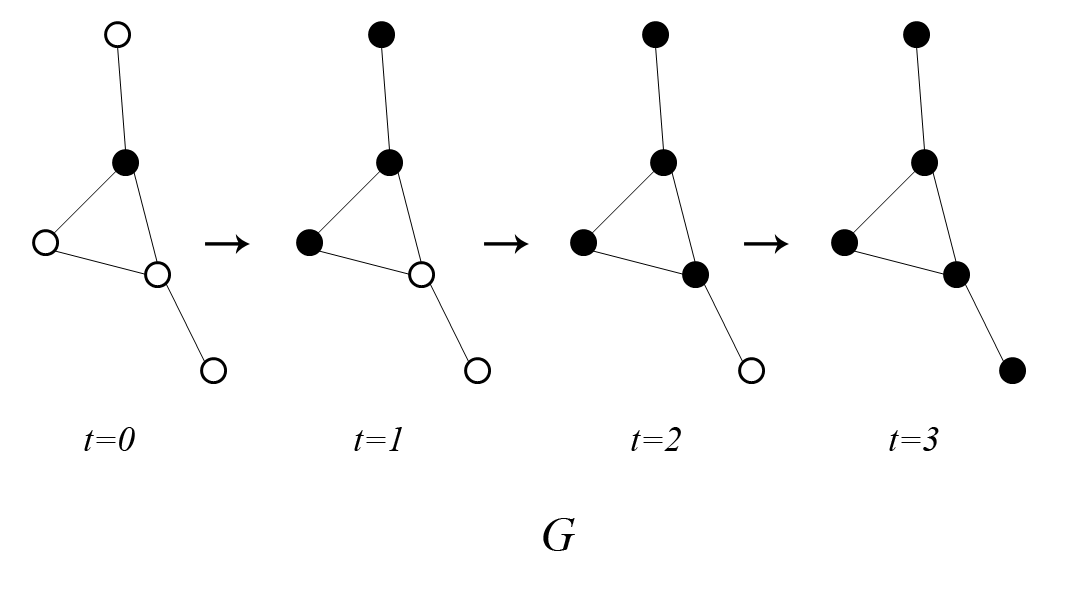}}
\hspace{10mm}
\mbox{\includegraphics[width=51mm]{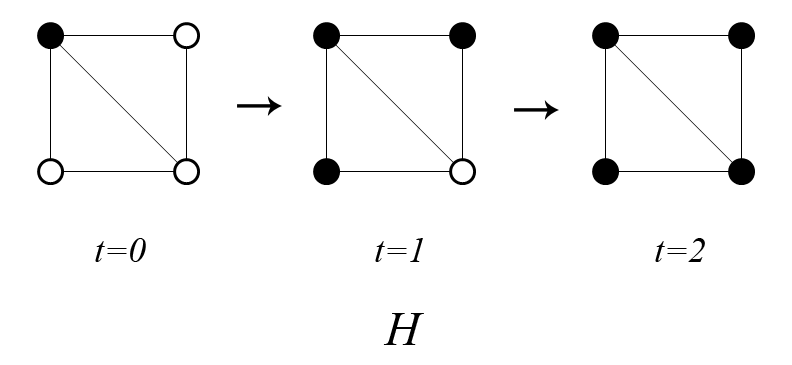}}}
\includegraphics[width=150mm]{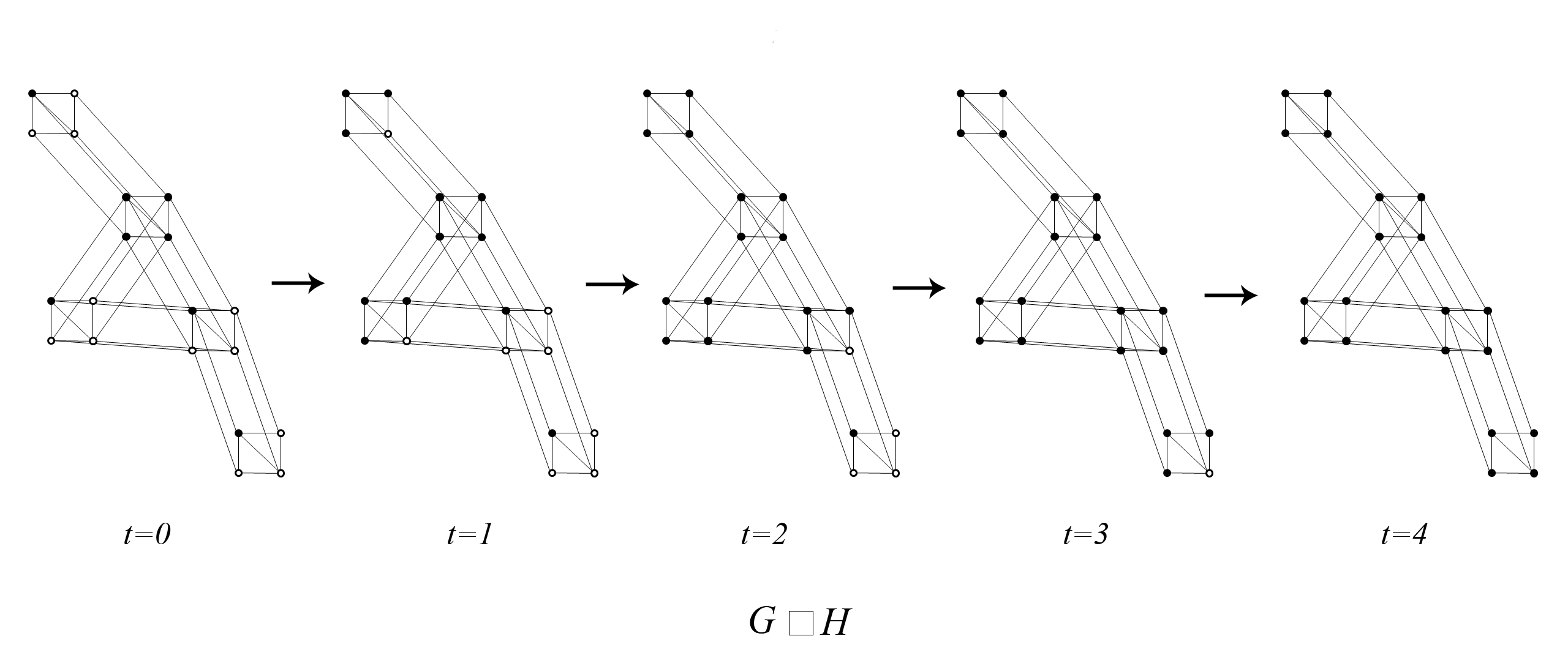}
\caption{Individual colorations of graphs $G$, $H$, $G\square H$ of Fig.~\ref{fig:Cartesianproduct} by dynamos $D_G$, $D_H$, $D$, resp.}
\label{fig:Gsteps}
\end{figure}

\begin{figure}
\begin{center}
\includegraphics[width=80mm]{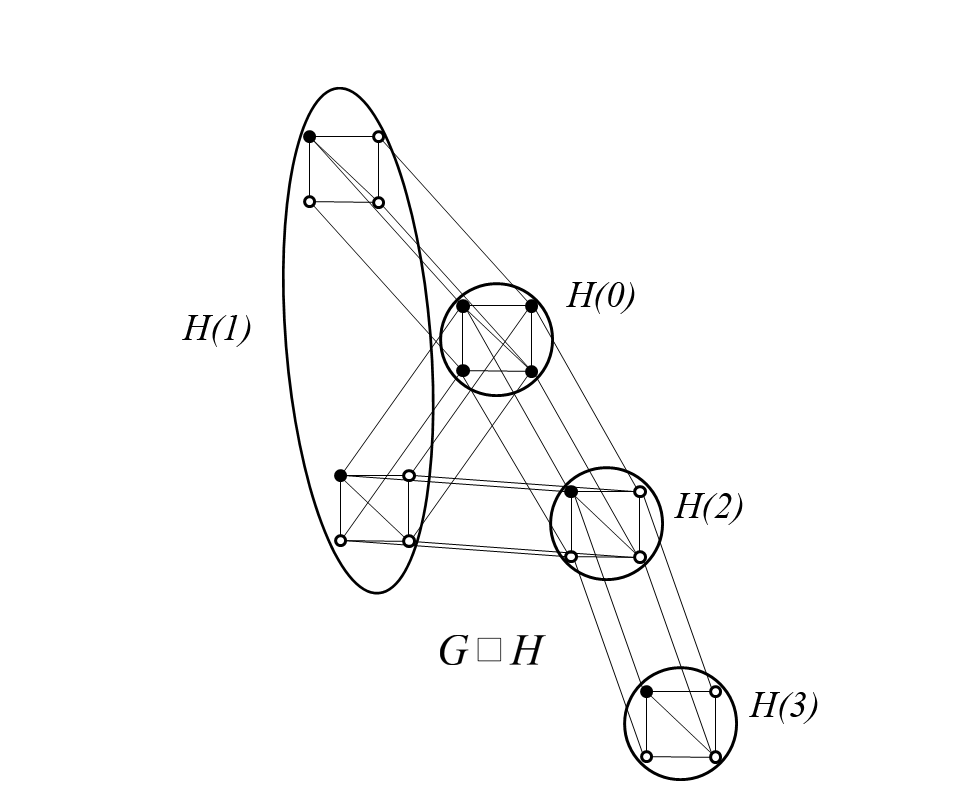}
\end{center}
\label{fig:subgraphs}
\caption{$V(G\square H)$ of Fig.~\ref{fig:Cartesianproduct} partitioned into $H(t)$, $t=0, 1, \ldots, T$ as in the proof of Theorem~\ref{thm:Cartesian}.}
\label{fig:Cartesian3}
\end{figure}

There exist graphs $G$ and $H$ so that the upper bound of $min_D(G\square H)$ provided in Theorem~\ref{thm:Cartesian} is tight. For example, if $G$ consists of $n$ isolated vertices and $H$ consists of $m$ isolated vertices, then $G\square H$ consists of $nm$ isolated vertices and the equality for the bound in Theorem~\ref{thm:Cartesian} will follow. If we require that $G$ and $H$ do not contain isolated vertices, Theorem~\ref{thm:Cartesian2} below offers a smaller upper bound of $min_D(G\square H)$. We must first prove Lemma~\ref{lem:comp} and Corollary~\ref{clry:half}.

\begin{lem}
\label{lem:comp}
Let $G$ be a graph without isolated vertices, and let $D$ be a minimal dynamo of $G$ (meaning that removing any vertex from $D$ results in a set that is no longer a dynamo). Then the set of vertices $V(G)-D$ is  a dynamo of $G$. In addition, the dynamo $V(G)-D$ colors all the vertices in $G$ by time step $t=1$.
\end{lem}
\begin{proof}
Initially color $D$. If a vertex $u$ in $D$ has $deg_D(u)\geq\dfrac{1}{2}deg_G(u)>0$ (recall that $G$ does not have isolated vertices), then $D-u$ would be a dynamo of $G$ because if all the vertices in $D-u$ were initially black and $u$ was initially white then its $deg_D(u)$ black neighbors in $D-u$ would cause $u$ to become black in the next time step. This would contradict the minimality of the dynamo $D$. Hence, $deg_D(u)<\dfrac{1}{2}deg_G(u)$ for all vertices $u$ in $D$. Note that this implies $D\neq V(G)$ since otherwise $deg_G(u)=deg_D(u)<\dfrac{1}{2}deg_G(u)$ for any $u$ in $D$, a contradiction.

Now, initially color $V(G)-D \neq \emptyset$ . Previously, we have shown that for each vertex $u$ in $D$, $deg_D(u)<\dfrac{1}{2}deg_G(u)$ which is equivalent to $deg_{V(G)-D}(u)>\dfrac{1}{2}deg_G(u)$, and therefore $u$ will be colored in the next time step. Thus, $V(G)-D$ is a dynamo of $G$ and all the vertices in $G$ become colored by time step $t=1$.
\end{proof}

\begin{clry}
\label{clry:half}
Let $G$ be a graph without isolated vertices.  Then, $min_{D}(G)\leq\dfrac{\left|V(G)\right|}{2}.$
\end{clry}
\begin{proof}
Let $D$ be a   minimum (hence minimal) dynamo of $G$. Suppose for contradiction that $min_D(G)=\left|D\right|>\dfrac{\left|V(G)\right|}{2}$. By Lemma \ref{lem:comp}, $V(G)-D \neq \emptyset$ is a dynamo of $G$ and $\left|V(G)-D\right| = \left|V(G)\right| - \left| D\right| <\left|V(G)\right| - \dfrac{\left|V(G)\right|}{2}=\dfrac{\left|V(G)\right|}{2}<\left|D\right| = min_D(G)$, contradicting the minimality of the dynamo $D$. Therefore, $min_{D}(G)\leq\dfrac{\left|V(G)\right|}{2}.$
\end{proof}

We note that the previous corollary was independently obtained by Chang and Lyuu~\cite{products:lyuu}, and also by Ackerman, Ben-Zwi and Wolfovitz~\cite{products:ack}, however our derivation is much simpler.

We are now prepared to present Theorem~\ref{thm:Cartesian2} which provides a smaller upper bound for $min_{D}(G \square H)$ than Theorem~\ref{thm:Cartesian} when  $G$ and $H$ do not contain isolated vertices.

\begin{thm}
Let $G$ and $H$ be two graphs without isolated vertices. Then,
\[  min_{D}(G \square H) \leq min_{D}(G)\left|V(H)\right| + min_{D}(H)\left|V(G)\right| - 2min_{D}(G)min_{D}(H).\]
\label{thm:Cartesian2}
\end{thm}
\begin{proof}
Let $D_{G}$ and $D_{H}$ be two minimum dynamos of $G$ and $H$, respectively, and let $D$ be the dynamo of $G\square H$ constructed in the proof of Theorem~\ref{thm:Cartesian} with $min_D(G) {\left|V(H)\right|}$+$min_D(H) {\left|V(G)\right|}-min_D(G)min_D(H)$ vertices. Let $D'$ be the set of $min_D(G)min_D(H)$ vertices $(u,v)$ in $D$ where $u$ and $v$ are in $D_{G}$ and $D_{H}$, respectively.  Using Corollary~\ref{clry:half} for $G$ and $H$, we have that
\begin{eqnarray*}
\left|D-D'\right| &=& min_D(G) {\left|V(H)\right|}+min_D(H) {\left|V(G)\right|}-2min_D(G)min_D(H) \\
&\geq& min_D(G)(2min_D(H))+min_D(H)(2min_D(G))-2min_D(G)min_D(H) \\
& =&  2 min_D(G)min_D(H) > 0,
\end{eqnarray*}
and therefore, $D-D'\neq\emptyset$. To verify the proposed bound, we will show that $D-D'$ is a dynamo of $G\square H$. Initially color $D-D'$. It then suffices to show that the vertices in $D'$ will eventually get colored since $D$ is a dynamo of $G \square H$.

Let $(u,v)$ be an arbitrary vertex in $D'$. By the definition of $D'$, we have that $u$ and $v$ are in $D_{G}$ and $D_{H}$, respectively. Note that $(u,v)$ belongs to $H_{u}$ and $G_{v}$ (recall Definition~\ref{defn:Cartesian}) and the vertices in the neighborhood of $(u,v)$ must be in exactly one of these two subgraphs, so $deg_{(G\square H)}(u,v)=deg_G(u)+deg_H(v)$.  From Lemma~\ref{lem:comp}, we have that $V(G)-D_G$ (resp., $V(H)-D_H$) is a dynamo of $G$ (resp.,  $H$)  which colors all the vertices in $G$ (resp., $H$) by time step $t=1$. This means that the vertex $u$ (resp. $v$) has at least half of its neighbors in $V(G)-D_G$ (resp., $V(H)-D_H$). Therefore the vertex $(u,v)$ has at least half of its neighbors in $D-D'$ which was initially colored and we can conclude that $(u,v)$ will become colored in the next time step. This shows that $D-D'$ is also a dynamo of $G\square H$ as desired.
\end{proof}

\section{ Tight upper bounds of $min_k(G\times H)$ and $min_D(G\times H)$  of the tensor product $G\times H$ } \label{tensor}
Recall the definition of the tensor product of two graphs.
\begin{defn}
The  \emph{tensor product} of two disjoint graphs $G$ and $H$ is the graph $G\times H$  such that $V(G\times H)=V(G)\times V(H)$ and two vertices $(u,u')$ and $(v,v')$ are adjacent if and only if $u$ is adjacent to $v$ in $G$ and $u'$ is adjacent to $v'$ in $H$.
\end{defn}
Note that each of the $i_G$ (resp., $i_H$) isolated vertices of $G$ (resp., $H$) generates $|V(H)|$ (resp., $|V(G)|$) isolated vertices in $G\times H$, and these are the only isolated vertices in $G\times H$. Thus, $G\times H$ contains $i_G|V(H)| + i_H|V(G)| - i_G i_H$ isolated vertices which must be all contained in any $k$-conversion set and dynamo of $G\times H$. We can then focus on determining upper bounds of $min_k(G\times H)$ and $min_D(G\times H)$ when $G$ and $H$ are graphs without isolated vertices and later add the number of potential isolated vertices to these bounds. We will first provide an upper bound of $min_k(G\times H)$.

\begin{thm}
Let $G$ and $H$ be two graphs without isolated vertices. Then,
\[ min_k(G \times H) \leq min\{ min_k(G)\left|V(H)\right|, min_k(H)\left|V(G)\right|\}.\]
\label{thm:ktensor}
\end{thm}
\begin{proof}
We may assume without loss of generality that $min_{k}(G)\left|V(H)\right|\leq min_{k}(H)\left|V(G)\right|$ because the tensor product is commutative.  Let $S_{G}$ be a minimum $k$-conversion set of $G$, and let $S$ be the set of $min_{k}(G)\left|V(H)\right|$ vertices $(u,v)$ in $G\times H$ where $u$ is in $S_{G}$.  To verify the proposed upper bound, it suffices to show that $S$ is a $k$-conversion set of $G\times H$.

We will first partition the vertices in $G \times H$ based on the coloration of $G$ by $S_G$. Suppose it takes $T$ time steps to fully color $G$ by $S_G$ and for each $t=0, 1, \ldots, T$, let $H(t)$ be the set of vertices $(u,v)$ where $u$ is a vertex that becomes colored at time step $t$ in $G$. The sets $H(0), H(1), \ldots, H(T)$ partition the vertices in $G\times H$.

Initially color $S$. We will show that all vertices in $G \times H$ will eventually be colored by proving that each $H(t)$ for $t=0, 1, \ldots, T$ will eventually be colored. We proceed by induction on $t$. As the base case, note that every vertex in $H(0)$ has its first coordinate in $S_{G}$ and consequently belongs to $S$ which is assumed to be initially colored. Now suppose that all of the vertices in $H(0), H(1), \ldots, H(t-1)$ have been colored and let us show that the vertices in $H(t)$ will become colored. Consider an arbitrary uncolored vertex $(u,v)$ in $H(t)$.  Note that $deg_{(G\times H)}(u,v)=deg_G(u)deg_H(v)$. Since $u$ is colored in time step $t$ in $G$ by $S_G$, at least $k$ of its $deg_G(u)$ neighbors in $G$ were black by time step $t-1$. Each one of these black neighbors in $G$ will generate $deg_H(v) > 0$ (recall that $H$ does not have isolated vertices) neighbors of $(u,v)$ in $G\times H$ in the union of $H(i)$, $i=0, 1, \ldots, t-1$ which contains only colored vertices. Therefore, at least $k$ of the $deg_G(u)deg_H(v)$ neighbors of $(u,v)$ in $G\times H$ are colored, and hence, $(u,v)$ will become colored. It now follows by induction that $S$ is a $k$-conversion set of $G \times H$ providing the desired upper bound of $min_k(G\times H)$.
\end{proof}

If in the proof of Theorem~\ref{thm:ktensor} we replace each occurrence of $min_k$, $S_G$, $S$, "$k$-conversion set", and "at least $k$" with, respectively, $min_D$, $D_G$, $D$, "dynamo," and "at least half," we obtain a proof for the analogous result under a majority conversion process stated in Theorem~\ref{thm:tensor} below. We omit the details in this proof but provide a concrete example in Fig.~\ref{fig:tensor1} to illustrate the construction of such dynamo $D$ of $G \times H$ where $V(G\times H)$ is partitioned into $H(t)$, $t=0, 1, \ldots, T$.

\begin{thm}
Let $G$ and $H$ be two graphs without isolated vertices. Then,
 \[ min_{D}(G\times H)\leq min\left\{min_{D}(G)\left|V(H)\right|,min_{D}(H)\left|V(G)\right|\right\}.\]
\label{thm:tensor}
\end{thm}

\begin{figure}
\includegraphics[width=80mm]{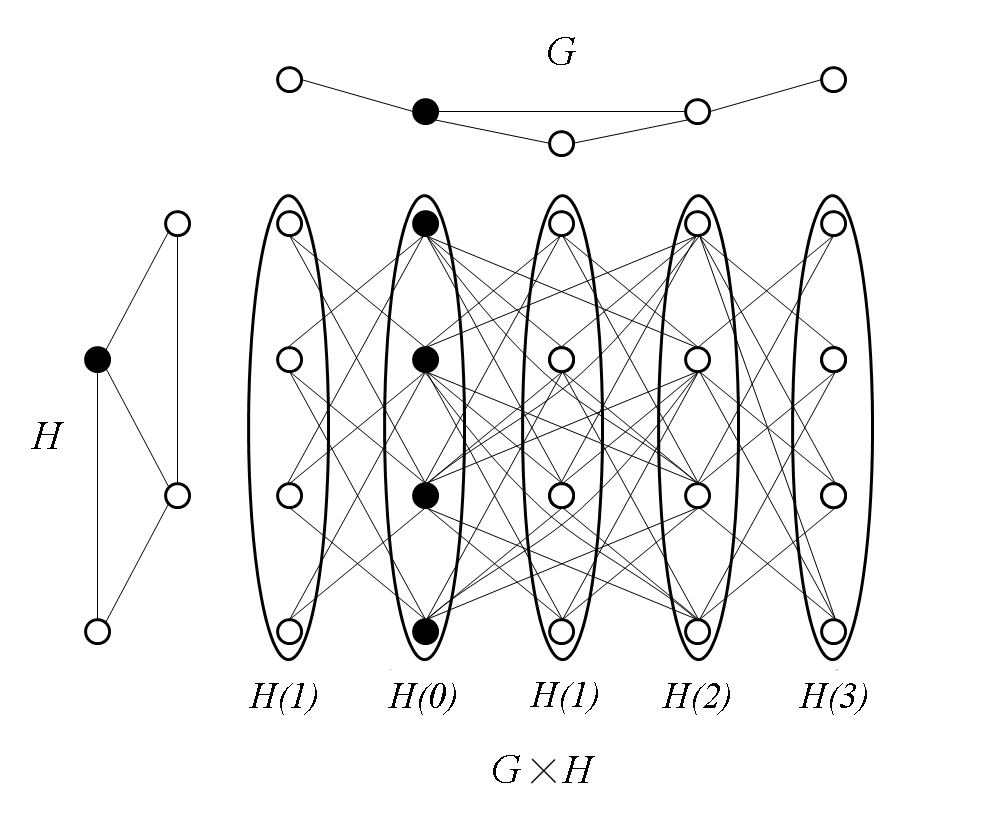}
\caption{Graphs $G$, $H$, $G\times H$ with minimum dynamos $D_G$, $D_H$, $D$, resp. (dynamos are colored black), where $\left| D\right| = min_D(G) {\left|V(H)\right|}$ for Theorem \ref{thm:tensor}; $V(G\times H)$ partitioned into $H(t)$, $t=0, 1, \ldots, T.$}
\label{fig:tensor1}
\end{figure}

We will close this section by providing an infinite family of graphs where the upper bounds in Theorems~\ref{thm:ktensor} and \ref{thm:tensor} are tight.  We will use the \emph{bipartite double cover} of a graph $G$ which is defined as the tensor product $G\times K_2$ where $K_2$ is the complete graph on two vertices. We state without proof the following lemma due to Sampathkumar~\cite{products:doublecover} which will be instrumental in our discussion.
\begin{lem} \rm{\cite{products:doublecover}}
The bipartite double cover of any connected bipartite graph $G$ is isomorphic to the graph $2G$, that is, two disjoint copies of $G$.
\label{lem:doublecover}
\end{lem}
Lemma~\ref{lem:doublecover1}, below, easily follows from Lemma~\ref{lem:doublecover}.
\begin{lem}
Let $G$ be a connected bipartite graph. Then $min_{k}(G\times K_{2})=2min_{k}(G)$ and $min_{D}(G\times K_{2})=2min_{D}(G)$.
\label{lem:doublecover1}
\end{lem}

From Lemma~\ref{lem:doublecover1}, one can verify the equalities in Theorems~\ref{thm:ktensor}  and Theorem~\ref{thm:tensor} if  $G$ is any connected bipartite graph with at least two vertices and $H=K_2$. We leave the details to the reader.

\section{Conclusions and Future Work}\label{conclude}

In this work, we have provided exact values of $min_k(G)$ and $min_D(G)$ when $G$ is a complete multipartite graph. We have also found upper bounds of these graph invariants when $G$ is the Cartesian and tensor products of two graphs but these bounds are not always tight. A natural next step is to improve on these general upper bounds for the Cartesian and tensor products within specific families of graphs and ultimately to obtain exact values.  Our forthcoming work concerns the remaining open cases for the Cartesian and tensor products of two cycles, of two paths, and of a path and a cycle.

\section*{Acknowledgements}

The authors would like to thank Harold Jaffe and Paul Booth for their initial work on related problems and S. Luke Zinnen for his careful proof-reading.  The authors also thank the National Science Foundation (NSF) for its support thorough grant EMSW21-MCTP 0636528. Denise Sakai Troxell would also like to thank the Babson Faculty Research Fund for its support.

%\bibliography{products}
%\bibliographystyle{plain}

\end{document}